\documentclass[review]{elsarticle}

\usepackage{lineno,hyperref}
\usepackage{mathrsfs}
\usepackage{amsmath,amssymb}

\newtheorem{theorem}{Theorem}[section]
\newtheorem{lemma}{Lemma}[section]
\newtheorem{remark}{Remark}[section]
\newenvironment{proof}[1][Proof]{\noindent\textbf{#1.} }{\hfill $\Box$}
\modulolinenumbers[5]
\journal{Mathematical Analysis and Applications}

\begin{document}

\begin{frontmatter}

\title{Mathematical analysis of an age structured problem modeling phenotypic plasticity in mosquito behaviour}

\author{Lin Lin Li}
  \address{Institut de Math\'ematiques de Bordeaux, Universit\'e de Bordeaux, Bordeaux, France}
  \ead{linlin.li@u-bordeaux.fr}

\author{Cl\'audia Pio Ferreira}
\address{S\~ao Paulo State University (UNESP), Institute of Biosciences, Department of Biostatistics, Botucatu, Brazil}
%Departamento de Bioestat\'{\i}stica,   IBB, Unesp, Botucatu-SP, Brasil}
   \ead{pio@ibb.unesp.br}
\author{Bedreddine Ainseba \corref{mycorrespondingauthor}}
\cortext[mycorrespondingauthor]{Corresponding author}
\ead{bedreddine.ainseba@u-bordeaux.fr}

\address[mymainaddress]{Institut de Math\'ematiques de Bordeaux, Universit\'e de Bordeaux, Bordeaux, France}

\begin{abstract}
This paper presents an age structured problem modelling mosquito blood-feeding plasticity  in a natural environment. We first investigate the analytical asymptotic solution through studying the spectrum of an operator $\mathbb{A}$ which is the infinitesimal generator of a $C_0$-semigroup. Indeed, the study of the spectrum of $\mathbb{A}$ {\it per se} is interesting. Additionally, we get the existence and nonexistence of nonnegative steady solutions under some conditions.
\end{abstract}


\end{frontmatter}

\linenumbers

\section{Introduction}

Malaria is an infectious disease caused by a species of parasite that belongs to the genus Plasmodium. This pathology
affects millions of people over the world, being predominant in equatorial region, e.g., Amazon rainforest, sub-saharan Africa and
South East Asia. The Plasmodium is transmitted by female Anopheles mosquitoes when they bite and, thus, feed on human
blood. Control mechanims acting on disease dynamics take into account the behaviourally characteristics of mosquito population, such as anthrophagy, endophily, endophagy, physiological susceptibility to pyrethroids, and night-biting preference.  The recent reports on Malaria transmission shown that the long-term use of residual spraying (IRs) and insecticide-treated nets (ITNs) has been driving mosquito physiological and behavioural resistance. Many mosquito species exhibit high levels of phenotypic plasticity that can be expressed on host preference, biting activity, etc.  Such heritable phenotypic plasticity allows individuals mosquitoes to flexibly adapt their behaviour according to the environmental conditions. The development of a crepuscular, outdoor feeding phenotype among anopheline population has been observed in areas of intensive use of IRs and ITNs. This adjust on biting time can jeopardize the sucesseful of Malaria control and promotes parasite evolution \cite{Gatton2013}.  Hence, researches on the population dynamics of mosquitoes become essential.

In this paper, we propose a partial differential equation system to model the plasticity of mosquitoes  in a natural environment, namely without any intervention of human activities, such as IRs and ITNs.
Let $p(a,t,x)$ be the distribution of individual mosquitoes of age $a\ge 0$ at time $t\ge 0$ and biting activity at time $x \in [0,24] $. The introduction of the variable $x$ in the system has the objective of illustrate mosquito biting behaviour, which will be of great importance in the following research on mosquitoes control.
Let $a_{\dagger }$ be the life expectancy of an individual mosquito and $T$ be a positive constant.
Let $\beta (a)\ge 0$ be the natural fertility-rate which is bounded, nonnegative and measurable on $[0,a_{\dagger }]$, and $\mu (a)\ge 0$ be the natural mortality-rate of mosquitoes of age $a$. The new generation of mosquitoes can adapt to ensure its survival and reproduction, changing the biting time in order to maximize its fitness.
This is modeled by the kernel function $K(x,s)$ in the renewal equation. The parameter $\eta$ is the maximum difference on biting time that the new generation can reach. Mosquitoes can change their biting time up to a diffusive coefficient $\delta$.
Thus, the evolution of the distribution $p(a,t,x)$ is governed by the system
\begin{equation}\label {eq:2929}
\left\{ \begin{array}{lll}
Dp- \delta \Delta _{x}p+\mu (a)p = 0,&(a,t,x)\in Q_{a_\dagger}, \\
p(a,t,0) =p(a,t,24),&(a,t)\in (0,a_\dagger)\times (0,T), \\
\partial_x p(a,t,0) =\partial_x  p(a,t,24),&(a,t)\in (0,a_\dagger)\times (0,T), \\
\displaystyle p(0,t,x)=\displaystyle \int_0^{a_\dagger} \beta(a)\displaystyle \int_{x-\eta}^{x+\eta}K(x,s)p(a,t,s)ds da,&(t,x)\in (0,T)\times (0,24), \\
\displaystyle p(a,0,x)=p_0(a,x) , &(a,x)\in (0,{a_\dagger} ) \times (0,24).
          \end{array}
  \right. \end{equation}
where $Q_{{a_\dagger}}=(0,{a_\dagger} ) \times (0,T) \times (0,24)$, $\Delta _{x}p(a,t,x)= \partial_{xx}  p(a,t,x)$, the kernel
\begin{equation*}\label{eq:02}
K(x,s)= \begin{cases}
   {(x-s)^{2}}{e^{-(x-s)^{2}}}, &s \in (0,24), \\
   0,  &\text{else},
   \end{cases}
\end{equation*}
and $D p(a,t,x)$ is the directional derivative of $p$ with respect to  direction $\left(
1,1,0\right)$, that is,
\[
Dp(a,t,x)={\lim_{\varepsilon \rightarrow 0} }\frac{p\left(
a+\varepsilon ,t+\varepsilon, x \right) -p\left( a,t,x\right) }{\varepsilon }.
\]
For  smooth enough $p$, it is obvious that
\[
Dp=\frac {\partial p}{\partial t}+\frac {\partial p}{\partial a}.
\]

Notice that in our model, the boundary condition is assumed to be periodic and the fertility term is nonlocal with the kernel $K(x,s)$. In fact, both  Dirichlet boundary condition and local fertility term are very popular in mathematical modeling, such as dynamics population models of a single species with age dependence and spatial structure. We now review some known results about such models, that is, replacing the periodic boundary condition and the fertility term by the Dirichlet condition and $\displaystyle \int_0^{a_\dagger} \beta(a)p(a,t,x)da$ respectively. Chan and Guo \cite{Chan} considred this model in the semigroup framework, by setting the fertility-rate $\beta$ and the mortality-rate $\mu$ being independent of the space variable $x$. They identified the infinitesimal generator and studied its spectral properties, which could be used to get the asymptotic behavior of the solutions. Then, Guo and Chan \cite{Guo} removed the independence setting of $\beta$, $\mu$ and  got the asymptotic expression of the solution by analyzing the spectrum of the infinitesimal generator. We also refer to the works of Langlais \cite{Langlais}, for the study of the long-time behaviour of the model where $\beta$ and $\mu$ depend on the distribution $p$.  The controllability problems on this model are also very attractive. Ainseba and Ani\c{t}a \cite{Ains3, As} studied the local exact controllability of such model with the Dirichlet boundary condition and the local fertility term. The control problem with Neumann boundary condition can be referred to \cite{Ains1,Ains2}.

We are interested in the ways on which Guo and Chan \cite{Chan,Guo} studied the asymptotic behaviour of the population model in \cite{Chan,Guo} throught the  analysis of the spectrum of the infinitesimal generator and using some positive semigroup theories.  In this paper, we mainly focus on the asymptotic behavior in Section 2. The key step for our paper is to find, for any initial $p_0(a,x) \in D(\mathbb{A}) $, the asymptotic expression $p(a,t,x)$.

Before presenting our results, we  introduce some usefull notations. Consider $X={L^{2}((0,{a_{\dagger}}) \times (0,24))}$ with the usual norm, and the operator $\mathbb{A}:X \longrightarrow X$  defined as
\begin{equation}\label{eq:2}
\mathbb{A}\phi(a,x)=- \frac{\partial\phi(a,x) }{\partial a}+\delta\Delta \phi(a,x)-\mu(a)\phi(a,x),\forall\phi(a,x)\in D(\mathbb{A}),
\end{equation}
where
\begin{align}\label{eq:3}
D(\mathbb{A})=\{&\phi(a,x)| \phi,\mathbb{A}\phi \in X, \phi(a,0)=\phi(a,24),\partial_x \phi(a,0) =\partial_x  \phi(a,24), \nonumber\\
&\phi(0,x)=\displaystyle \int_0^{a_\dagger} \beta(a)\displaystyle \int_{x-\eta}^{x+\eta}K(x,s)\phi(a,s)ds  da\}.
\end{align}
From the definition of the operator $\mathbb{A}$,  the system \eqref{eq:2929} can be transformed into an evolutionary equation on the space $X$:
\begin{eqnarray*}
\left\{\begin{aligned}
&\frac{d p(a,t,x)}{dt}=\mathbb{A}p(a,t,x),\\
&p(a,0,x)=p_{0}(a,x).
\end{aligned}
\right.
\end{eqnarray*}

For the following notations, we can refer to Marek \cite[p.609]{IVO} and Clement \cite{Clement} for instance. If $\mathbb{A}$ is a linear operator from $X$ into $X$, then $\rho(\mathbb{A})$ denotes the resolvent set of $\mathbb{A}$, that is, $\rho(\mathbb{A})$ is the set of all complex numbers $\lambda$ for which $(\lambda \mathbb{I}-\mathbb{A})^{-1}$ is a bounded automorphism of $\mathbb{A}$ (let $R(\lambda,A)=(\lambda \mathbb{I}-\mathbb{A})^{-1}$ called the resolvent operator), where $\mathbb{I}$ denotes the identity operator. The complement of $\rho(\mathbb{A})$ in the complex plane is the spectrum of $\mathbb{A}$, and it is denoted by $\sigma(\mathbb{A})$. We denote by $\gamma(\mathbb{A})$ the spectral radius of $\mathbb{A}$, that is,
$$\gamma(\mathbb{A})=\sup\{|\lambda|:\ \lambda\in \sigma(\mathbb{A})\}.$$
If $\mathbb{A}$ is an infinitesimal generator of a $C_{0}$-semigroup $T(t)$ on the space $X$, the spectral bound $s(\mathbb{A})$ can be denoted by
$$s(\mathbb{A})=\sup\{|\lambda|:\ Re\lambda\in \sigma(\mathbb{A})\}.$$
 And the growth bound of the semigroup $T(t)$ can be shown as
$$\omega(\mathbb{A})=\inf_{t>0} \frac{1}{t} \log \|T(t)\|_{L^2((0,a_\dagger)\times(0,24))}=\underset{t \rightarrow +\infty}\lim \frac{1}{t} \log \|T(t)\|_{L^2((0,a_\dagger)\times(0,24))}.$$

Let $a_{\dagger}$ be a finite positive number. From the biological point of view addressed in \cite{Garr,Gurt,Webb}, we assume the following assumptions throughout this paper:
\begin{description}
\item[(J1)]$\mu(a)\in L^{1}_{loc}([0,{a_{\dagger}}))$ and $\displaystyle \int_0^{a_\dagger} {\mu}(\rho)d \rho = \infty $;
\item[(J2)]$\beta(a) \in L^{\infty}((0,{a_{\dagger}}))$, $\text{mes}\{a| a \in [0,{a_{\dagger}}],\beta(a)>0\}>0$;
\item[(J3)]$p_{0}(a,x)\in L^{\infty}((0,{a_\dagger} ) \times (0,24))$, $p_{0}(a,x)\geq 0$.
\end{description}

The following theorems are the main results of our paper and they will be proved in the following sections.
\begin{theorem}\label{th2}
For any initial $p_0(a,x) \in D(\mathbb{A}) $, the semigroup solution of \eqref{eq:2929}  has the following asymptotic expression:
\begin{align*}
p(a,t,x)=&e^{\lambda_{0}t}e^{-\lambda_{0}a}\mathscr{T}(0,a)C_{\lambda_{0}}\displaystyle \int_0^{a_\dagger} \beta(a)\displaystyle \int_{x-\eta}^{x+\eta}K(x,s)\displaystyle \int_0^{a}e^{-\lambda_{0}(a-\sigma)}\mathscr{T}(\sigma,a)\\&p_{0}(\sigma,s)dsdad\sigma+o(e^{(\lambda_{0}-\epsilon)t}),
\end{align*}
where $\lambda_{0}$, $C_{\lambda_{0}}$and $\mathscr{T}(\tau,s)$ will be defined in Section 2.
\end{theorem}

The steady state of our model is very important, especially for our further researches about the control problem. The steady state of \eqref{eq:2929} is denoted by $p_{s}$, and should be a solution of
\begin{equation}\label {eq:04}
\left\{ \begin{array}{lll}
\partial_a p_{s}(a,x)- \delta \Delta p_{s}(a,x)+\mu (a)p_{s}(a,x)= 0,&(a,x)\in (0,a_\dagger)\times (0,24), \\
p_{s}(a,0)=p_{s}(a,24),&a\in (0,a_\dagger), \\
\partial_x p_{s}(a,0) =\partial_x  p_{s}(a,24),&a\in (0,a_\dagger), \\
p_{s}(0,x)= \int_0^{a_\dagger} \beta(a) \int_{x-\eta}^{x+\eta}K(x,s)p_{s}(a,s)ds da,&x\in (0,24). \\
\end{array}
\right. \end{equation}
Furthermore, $p_{s}(a,x)$ satisfies
\begin{equation}\label {eq:05}
p_{s}(a,x)\geq 0 \ \ a.e. \  (a,x)\in (0,{a_{\dagger}})\times (0,24).
\end{equation}

\begin{theorem} \label{th3}
Consider \eqref{eq:04} with  $\lambda_{0}$ satisfying  Theorem \ref{th2}.
\begin{description}
\item[(1)] If $\lambda_{0}>0$, then there is no nonnegative solution of \eqref{eq:04} satisfying \eqref{eq:05}.
\item[(2)] If $\lambda_{0}=0$, then there exists infinitely many nontrivial solutions of \eqref{eq:04} satisfying \eqref{eq:05}. Furthermore, for any nonzero steady state $p_s(a,x)$, there exists $\rho_0>0$ such that
$$p_s(a,x)\ge \rho_0>0,\ \text{a.e. }(a,x)\in(0,a_1)\times(0,24),$$
where $a_1\in(0,a_{\dagger})$.
\item[(3)] If $\lambda_{0}<0$, then only trivial solutions $p_{s}$ of \eqref{eq:04} satisfying \eqref{eq:05} exist,
that is
\begin{equation*}
p_{s}(a,x)=0 \ \ a.e. \  (a,x)\in (0,{a_{\dagger}})\times (0,24).
\end{equation*}
\end{description}
\end{theorem}

The rest of this paper is organized as follows.
In Section 2, we make some preparations
which are necessary in what follows and we prove that $\mathbb{A}$ is an infinitesimal generator of a $C_{0}$-semigroup $T(t)$.
In section 3, we get the asymptotic behavior of \eqref{eq:2929} by analyzing the spectrum of the semigroup $T(t)$. Many abstract theories about semigroups used in this section can be referred to \cite{Clement,Yu,SAW}.
According to the asymptotic behaviour, we investigate the existence of steady states in Section 4.

\section{Preliminaries}
\noindent
In this section, we give some auxiliary lemmas as a preparation for
our main results that will be derived later. In fact, we have to prove that $\mathbb{A}$ is an infinitesimal generator of a $C_{0}$-semigroup $T(t)$.

At the beginning of this section, we study the following system
\begin{equation}\label{2.1}
\left\{ \begin{array}{lll}
Dp- \delta \Delta _{x}p+\mu (a)p = 0,&(a,t,x)\in Q_{a_\dagger},\\
p(a,t,0) =p(a,t,24),&(a,t)\in (0,a_\dagger)\times (0,T),\\
\partial_x p(a,t,0) =\partial_x  p(a,t,24),&(a,t)\in (0,a_\dagger)\times (0,T),\\
\displaystyle p(0,t,x)=\displaystyle C\int_0^{a_\dagger} \beta(a)\displaystyle p(a,t,x) da,&(t,x)\in (0,T)\times (0,24),\\
\displaystyle p(a,0,x)=p_0(a,x) , &(a,x)\in (0,{a_\dagger} ) \times (0,24),
\end{array}
\right.
\end{equation}
where $C$ can be any constant.
Defining the operator $\mathbb{F}:X\rightarrow X$ as:
\begin{equation}\label{2.2}
\mathbb{F}\phi(a,x)=- \frac{\partial\phi(a,x) }{\partial a}+\delta\Delta \phi(a,x)-\mu(a)\phi(a,x),\forall\phi(a,x)\in D(\mathbb{F}),
\end{equation}
where
\begin{align*}
D(\mathbb{F})=\{&\phi(a,x)| \phi,\mathbb{A}\phi \in X, \phi(a,0)=\phi(a,24),\partial_x \phi(a,0) =\partial_x  \phi(a,24), \nonumber\\
&\phi(0,x)=C\displaystyle \int_0^{a_\dagger} \beta(a)\displaystyle \phi(a,x)  da\},
\end{align*}
we can rewrite \eqref{2.1} as
\begin{eqnarray*}
\left\{\begin{aligned}
&\frac{d p(a,t,x)}{dt}=\mathbb{F}p(a,t,x),\\
&p(a,0,x)=p_{0}(a,x).
\end{aligned}
\right.
\end{eqnarray*}
Define an operator
\begin{equation}\label{eq66}
\mathscr{F}_{\lambda}= \int_0^{a_{\dagger}}C\beta(a)e^{-\lambda a}e^{-  \int_0^{a}{\mu}(\rho)d \rho} e^{\mathbb{B}a}d a,
\end{equation}
where the operator $\mathbb{B}:{L^{2}((0,24)) \longrightarrow L^{2}((0,24))}$ is defined as
\begin{align*}
\mathbb{B}u(x)=\delta\Delta u(x),
\end{align*}
for $u(x)$ satisfying
\begin{equation*}
\begin{cases}
&u(0)=u(24),\\
&u'(0)=u'(24).
\end{cases}
\end{equation*}

\begin{lemma}\label{lemma2.1}
The operator $\mathbb{F}$ defined by \eqref{2.2}.
\item[(1)] $\mathbb{F}$ has a real dominant eigenvalue $\widetilde{\lambda}_0$, that is, $\widetilde{\lambda}_0$ is greater than any real parts of the eigenvalues of $\mathbb{F}$.
\item[(2)] For the operator $\mathscr{F}_{\widetilde{\lambda}_0}$, $1$ is an eigenvalue with an eigenfunction $\phi_0(x)$. Furthermore, $\gamma(\mathscr{F}_{\widetilde{\lambda}_0})=1$.
\end{lemma}

\begin{proof}
(1) We denote by $(\overline{\lambda}_i,\phi_i)_{i\ge 0}$ the eigenvalues and the eigenfunctions of the following problem
\begin{eqnarray*}
\left\{\begin{aligned}
&-\delta \Delta \phi_i(x)=\overline{\lambda}_i \phi_i(x),\ x\in (0,24),\\
&\phi_i(0)=\phi_i(24),\\
&\partial_x \phi_i(0) =\partial_x  \phi_i(24),
\end{aligned}
\right.
\end{eqnarray*}
where
$\int_{0}^{24} \phi_i^2(x) dx=1,\ i\ge 0$,
and
$\ \phi_0(x)>0\ \text{with $x\in (0,24)$}$.
It is obvious that $\overline{\lambda}_0=0$ and $\phi_0(x)$ is a fixed positive constant.
We also assume that $0= \overline{\lambda}_0<\overline{\lambda}_1\le \overline{\lambda}_2\le \cdots$.

Let $F$ be the operator in $L^2(0,a_{\dagger})$ defined as
$$
F\phi(a)=-\frac{d \phi(a)}{d a} -\mu(a) \phi(a),\ \forall \phi\in D(F),$$
where
$$
D(F)=\{\phi(a)| \phi, F\phi \in L^2(0,24), \phi(0)=C\int_0^{a_\dagger} \beta(a) \phi(a)da\}.
$$
Let $\{\hat{\lambda}_j\}_{j\ge 0}$ be the eigenvalues of $F$, that is, the solutions of the following equation
\begin{align}\label{eq2.4}
1-C\int_0^{a_{\dagger}} \beta(a) e^{-\hat{\lambda}_j a-\int_0^{a} \mu(\rho)d\rho}da=0.
\end{align}
We assume that $\hat{\lambda}_0>Re \hat{\lambda}_1\ge Re \hat{\lambda}_2\ge \cdots$, even if it means re-arrange $\hat{\lambda}_j$.

Now, we divide two steps to consider the following equation
\begin{align}\label{2.4}
(\lambda \mathbb{I}-\mathbb{F})\phi =\psi,\ \forall \psi\in X.
\end{align}

Step 1, for any $i$, $j\ge 0$, $\lambda+\overline{\lambda}_i\neq \hat{\lambda}_j$,  define
$$\phi(a,x)=\sum_{i=0}^{\infty} R(\lambda+\overline{\lambda}_i,F) \langle\psi(a,x),\phi_i(x)\rangle\phi_i(x),$$
where $\langle\psi(a,x),\phi_i(x)\rangle=\int_0^{24} \psi(a,x)\phi_i(x)dx$, $R(\lambda,F)=(\lambda \mathbb{I}-F)^{-1}$, the resolvent operator of $F$.
Firstly, we prove that $\phi(a,x)\in X$ is well defined.
Since $F$ is the infinitesimal generator of a bounded strongly continuous semigroup from \cite{Song}, there exist constants $M$, $\omega>0$ such that
$$\|R(\lambda,F)\|\le \frac{M}{Re \lambda -\omega}, \text{ for $Re \lambda>\omega$}.$$
Recalling that $\overline{\lambda}_i\rightarrow \infty$ as $i\rightarrow \infty$, there is $N$ such that $Re (\lambda+\overline{\lambda}_i)>\omega$ when $i>N$. Then, one can compute that
\begin{align*}
&\sum_{i=0}^{\infty} \|R(\lambda+\overline{\lambda}_i,F) \langle\psi(a,x),\phi_i(x)\rangle\|^2\\
\le& \sum_{i=0}^{N} \|R(\lambda+\overline{\lambda}_i,F) \langle\psi(a,x),\phi_i(x)\rangle\|^2 \\
&+\left[\frac{M}{Re (\lambda+\overline{\lambda}_N)-\omega}\right]^2 \sum_{i=N+1}^{\infty} \|\langle\psi(a,x),\phi_i(x)\rangle\|\\
\le& \sum_{i=0}^{N} \|R(\lambda+\overline{\lambda}_i,F) \langle\psi(a,x),\phi_i(x)\rangle\|^2 +\left[\frac{M}{Re (\lambda+\overline{\lambda}_N)-\omega}\right]^2  \|\psi\|^2\\
<&\infty.
\end{align*}
It implies that $\phi(a,x)\in X$ is well defined.
Secondly, we prove $\phi(a,x)$ is a solution of \eqref{2.4}.
For any $n>0$,
\begin{align*}
&(\lambda \mathbb{I}-\mathbb{F})\sum_{i=0}^n R(\lambda+\overline{\lambda}_i,F) \langle\psi(a,x),\phi_i(x)\rangle\phi_i(x)\\
=&\sum_{i=0}^n [\lambda R(\lambda+\overline{\lambda}_i,F) \langle\psi(a,x),\phi_i(x)\rangle\phi_i(x)-\mathbb{F}R(\lambda+\overline{\lambda}_i,F) \langle\psi(a,x),\phi_i(x)\rangle\phi_i(x)]\\
=&\sum_{i=0}^n [\lambda R(\lambda+\overline{\lambda}_i,F) \langle\psi(a,x),\phi_i(x)\rangle\phi_i(x)-{F}R(\lambda+\overline{\lambda}_i,F) \langle\psi(a,x),\phi_i(x)\rangle\phi_i(x)\\
&- R(\lambda+\overline{\lambda}_i,F) \langle\psi(a,x),\phi_i(x)\rangle \delta \Delta \phi_i(x)]\\
=&\sum_{i=0}^n ((\lambda+\overline{\lambda}_i)\mathbb{I} - F) R(\lambda+\overline{\lambda}_i,F) \langle\psi(a,x),\phi_i(x)\rangle\phi_i(x)\\
=&\sum_{i=0}^n \langle\psi(a,x),\phi_i(x)\rangle\phi_i(x)\\
\rightarrow & \psi(a,x), \ n\rightarrow \infty.
\end{align*}
Since $F$ and $\Delta$ are both closed operators on $X$, one can infer that $\mathbb{F}$ is closed. Hence, $(\lambda \mathbb{I}-\mathbb{F})\phi=\psi$, that is, $\phi(a,x)$ is a solution of \eqref{2.4}.
Furthermore, it can be shown that $\phi$ is the unique solution of \eqref{2.4}, and thus $\lambda\in \rho(\mathbb{F})$, the resolvent set of $\mathbb{F}$ and
$$R(\lambda,\mathbb{F})\psi=\sum_{i=0}^{\infty} R(\lambda+\overline{\lambda}_i,F) \langle\psi(a,x),\phi_i(x)\rangle\phi_i(x).$$

Step 2, for some $i$, $j$ such that $\lambda+\overline{\lambda}_i=\hat{\lambda}_j$, it is easy to check that
$$\phi(a,x)=e^{-(\lambda+\overline{\lambda}_i)a-\int_0^{a}\mu(\rho)d\rho} \phi_i(x)$$
satisfies $(\lambda \mathbb{I}-\mathbb{F})\phi=0$, that is, $\lambda=\hat{\lambda}_j-\overline{\lambda}_i\in\sigma(\mathbb{F})$. In particular, $\widetilde{\lambda}_0=\hat{\lambda}_0-\overline{\lambda}_0$ is the dominant eigenvalue of $\mathbb{F}$, with  eigenfunction
\begin{equation*}
\phi_{\widetilde{\lambda}_0}(a,x)=e^{-\hat{\lambda}_0-\int_0^{a}\mu(\rho)d\rho}\phi_0(x).
\end{equation*}

It is easy to check that $C\phi_0(x)$ is the eigenfunction of the eigenvalue $1$ of $\mathscr{F}_{\widetilde{\lambda}_0}$, where $\widetilde{\lambda}_0=\hat{\lambda}_0-\overline{\lambda}_0$. Let any $\phi(x)\in L^2(0,24)$ be expanded as
$$\phi(x)=\sum_{i=0}^{\infty} \alpha_i \phi_i(x).$$
Then,
\begin{align*}
\mathscr{F}_{\widetilde{\lambda}_0}\phi(x)=\sum_{i=0}^{\infty} \alpha_i \int_0^{a_{\dagger}}C\beta(a)e^{-\widetilde{\lambda}_0 a}e^{-  \int_0^{a}{\mu}(\rho)d \rho} e^{\mathbb{B}a}\phi_i(x)da\\
=\sum_{i=0}^{\infty} \alpha_i \int_0^{a_{\dagger}}C\beta(a)e^{-(\widetilde{\lambda}_0 +\overline{\lambda}_i)a}e^{-  \int_0^{a}{\mu}(\rho)d \rho}da \phi_i(x).
\end{align*}
Since $\overline{\lambda}_i\ge\overline{\lambda}_0$ and then $\widetilde{\lambda}_0+\overline{\lambda}_i\ge \hat{\lambda}_0$, it follows from \eqref{eq2.4} that
$$\int_0^{a_{\dagger}}C\beta(a)e^{-(\widetilde{\lambda}_0 +\overline{\lambda}_i)a}e^{-  \int_0^{a}{\mu}(\rho)d \rho}da\le 1.$$
Thus, $\gamma(\mathscr{F}_{\widetilde{\lambda}_0})=1$.
\end{proof}

Following the proof of lemma 1 in \cite{Guo} carefully, we can get the following lemma:
\begin{lemma}\label {lemma:01}
For any $0\leq s_{0}<a_{+}$, there exists a unique mild solution u(s,x), $0\leq \tau\leq a_{+}-s_{0}$
to the evolution equation on $X$ for any initial function $\phi(x)\in L^{2}((0,24))$
\begin{equation*}
\begin{cases}
\frac{\partial u(s,x)}{\partial s} =(-\mu(s_{0}+s)+\mathbb{B})u(s,x),\\
u(\tau,x)=\phi(x),
\end{cases}
\end{equation*}
where the operator $\mathbb{B}_{0}$ is considered to be the Laplace operator with periodic boundary condition.
Define solution operators of the initial value problem by
\begin{equation*}
\mathscr{T}(s_{0},\tau,s)\phi(x)=u(s,x),\ \ \ \ \forall\phi(x)\in L^{2}((0,24)),
\end{equation*}
then $\mathscr{T}(s_{0},\tau,s)\phi(x)$ is a family of uniformly linear bounded compact positive operators on $X$
and is strongly continuous about $\tau$,$s$.
Furthermore,
\begin{equation*}
\mathscr{T}(s_{0},\tau,s)=e^{-\int_\tau^{s}{\mu}(s_{0}+\rho)d \rho}e^{\mathbb{B}(s-\tau)},
\end{equation*}
where $e^{\mathbb{B}s}$ is the positive analytic semigroup generated by the operator $\mathbb{\mathbb{B}}$.
\end{lemma}
\begin{proof}
The proof is similar as that of lemma 1 in \cite{Guo}, so we omit the details here.
\end{proof}

\begin{lemma}\label{lemma3.2}
The operator $\mathbb{A}$ defined by \eqref{eq:2} and \eqref{eq:3} is the infinitesimal generator of a $C_{0}$-semigroup $T(t)$ on the space $X$.
\end{lemma}
\begin{proof}
First note that a $C_{0}$-semigroup $T(t)$ implies that there exists a constant $\omega$ and $M\geq1$, so that
\begin{align*}
\|T(t)\| \leq Me^{\omega t},\ \ \ \ \ \ \ \forall t\geq 0.
\end{align*}
Our strategy here is to apply the generalized Hille-Yoside Theorem (refer to Theorem 8.2.5 of \cite{Ye} and Corrollary 3.8 of \cite{Pazy}),
that is, to prove: (i) $\mathbb{A}$ is closed and $\overline{D(\mathbb{A})}=X$; (ii) for any $\lambda > \omega$, $\lambda \in \rho(\mathbb{A})$, and
\begin{align*}
\|R^{n}(\lambda,\mathbb{A})\| \leq \frac{M}{(\lambda-\omega)^{n}}, \ \ \ \ n=1, 2, 3 \cdots.
\end{align*}

(i) One can compute that
\begin{equation}\label {eq:1111}
\langle \mathbb{A}\phi(a,x),\phi(a,x)\rangle  \leq N  \int_0^{a_\dagger} \beta^{2}(a)d a \langle\phi(a,x), \phi(a,x)\rangle,
\end{equation}
for some constants $N>0$, which also implies that $\mathbb{A}$ is an m-dissipative operator when $\lambda \in \rho(\mathbb{A})$ for all sufficiently large $\lambda > 0$.
In fact, if this claim holds, $\mathbb{A}$ is a closed operator, and combining with the m-dissiptiveness of $\mathbb{A}$,
we know that, for all sufficiently large $\lambda, (\mathbb{A}-\lambda\mathbb{I})$ is dissipative and $R(\mathbb{I}-(\mathbb{A}-\lambda \mathbb{I}))$ equals the
whole space $X$. Thus from Theorem 4.6 in \cite{Pazy}, it follows that $D(\mathbb{A}-\lambda \mathbb{I})$ is dense in $X$ and so is
$D(\mathbb{A})$, since $X$ is a Hilbert space.

(ii) Now, we prove that $\lambda \in \rho(\mathbb{A})$ for all sufficiently large $\lambda > 0$.
In order to do this, we deal with the following equation
\begin{align*}
(\lambda \mathbb{I}-\mathbb{A})\phi(a,x)=\psi (a,x) ,\ \ \ \ \ \ \ \forall \psi \in  X,
\end{align*}
that is,
\begin{equation*}
\begin{cases}
\frac{\partial \phi(a,x)}{\partial a} =-(\lambda+\mu(a))\phi(a,x)+ \delta \Delta \phi(a,x)+\psi(a,x) ,\\
\phi(0,x)= \int_0^{a_\dagger} \beta(a) \int_{x-\eta}^{x+\eta}K(x,s)\phi(a,s)ds da.
\end{cases}
\end{equation*}
Let $\mathscr{T}(0,\tau,s)=\mathscr{T}(\tau,s)= e^{- \int_\tau^{s}{\mu}(\rho)d \rho}e^{\mathbb{B}(s-\tau)}$
and by Lemma \ref{lemma:01}, one has
\begin{align*}
\phi(a,x)=e^{-\lambda a}\mathscr{T}(0,a)\phi(0,x)+  \int_0^{a} e^{-\lambda(a-\delta)} \mathscr{T}(\delta,a)\psi(\delta,x)d \delta,
\end{align*}
and
\begin{align}\label {eq:09}
&\phi(0,x)- \int_0^{a_\dagger} \beta(a) \int_{x-\eta}^{x+\eta}K(x,s)e^{-\lambda a}\mathscr{T}(0,a)\phi(0,s)ds da \nonumber \\
&=  \int_0^{a_\dagger} \beta(a)  \int_{x-\eta}^{x+\eta}K(x,s)  \int_0^{a}e^{-\lambda (a-\delta)}\mathscr{T}(\delta,a)\psi(\delta,s)d \delta d s da.
\end{align}
Then define the operator $\mathscr{B}_{\lambda}: L^2((0,24))\rightarrow L^2((0,24))$ by
\begin{align}\label {eq:33}
\mathscr{B}_{\lambda}(\phi (x))=  \int_0^{a_\dagger} \beta(a)  \int_{x-\eta}^{x+\eta}K(x,s)e^{-\lambda a}\mathscr{T}(0,a)\phi(s)ds da.
\end{align}
Here, notice that $\mathscr{B}_{\lambda}(\phi (x))$ is nonlocal in $x$ with $\phi(x)$, since the part of the operation $\mathscr{B}_{\lambda}$ on $\phi(x)$ is the integral $\int_{x-\eta}^{x+\eta}K(x,s)\mathscr{T}(0,a)\phi(s)ds$.  This is different of \cite{Chan} and \cite{Guo}, whose related operators are local. Therefore, $\lambda \in \rho(\mathbb{A}) $ if and only if $1\in \rho(\mathscr{B}_{\lambda})$. Furthermore, it follows from \eqref {eq:09} and \eqref {eq:33} that
\begin{align*}
\phi(0,x)=&(\mathbb{I}-\mathscr{B}_{\lambda})^{-1}  \int_0^{a_\dagger} \beta(a)  \int_{x-\eta}^{x+\eta}K(x,s)\int_0^{a}e^{-\lambda (a-\delta)}\mathscr{T}(\delta,a)\psi(\delta,s)d \delta d s da,
\end{align*}
and
\begin{align}\label {eq:5555}
R(\lambda,\mathbb{A}) \psi(a,x)=&e^{-\lambda a}\mathscr{T}(0,a)( \mathbb{I}-\mathscr{B}_{\lambda})^{-1}  \int_0^{a_\dagger} \beta(a)
\int_{x-\eta}^{x+\eta}K(x,s) \int_0^{a}e^{-\lambda (a-\delta)}\nonumber\\
&\mathscr{T}(\delta,a)\psi(\delta,s)d \delta d s da +\int_0^{a}e^{-\lambda (a-\delta)}\mathscr{T}(\delta,a)\psi(\delta,x)d \delta.
\end{align}
By the definitions of $K(x,s)$ and $\mathscr{T}(0,a)$, we can show that
\begin{align*}
\|\mathscr{B}_{\lambda}\|\leq\|  \int_0^{a_\dagger} \beta(a)e^{-\lambda a}e^{-  \int_0^{a}\mu(\rho)d \rho }e^{\mathbb{B}a} da\|,
\end{align*}
which implies that
$$
\lim_{\lambda \rightarrow + \infty} \|\mathscr{B}_{\lambda}\|=0.
$$
Hence, for all sufficiently large $\lambda>0$, $(\mathbb{I}-\mathscr{B}_{\lambda})^{-1}$ exists and is bounded.
Thus $1\in \rho(\mathscr{B}_{\lambda})$ which is equivalent to $\lambda \in \rho(\mathbb{A})$.

From \eqref{eq:1111} , one can obtain
after some computations that
\begin{align*}
\|R^{n}(\lambda,\mathbb{A})\| \leq \frac{M}{(\lambda-\omega)^{n}}, \ \ \ \ n=1, 2, 3 \cdots.
\end{align*}
 This completes the proof.
\end{proof}

\section{Asymptotic behavior}
\noindent
In this section, we study the asymptotic behavior of solutions of \eqref{eq:2929} by analyzing the spectrum of the semigroup.
It means that we will prove Theorem \ref{th2}.

Now, we state the asymptotic expression which indicates the asymptotic behavior.
\begin{theorem}\label{66666}
\begin{description}\item[(1)]For the eigenvalues of the operator $\mathbb{A}$,
there is only one real eigenvalue $\lambda_{0}$ which is algebraically simple and is larger than any real part of the other eigenvalues. \item[(2)]The semigroup $T(t)$ has the asymptotic expression
  \begin{align*}
  T(t)\phi(a,x)=&e^{\lambda_0 t}e^{-\lambda_0 a} \mathscr{T}(0,a) C_{\lambda_0} \int_0^{a_{\dag}} \beta(a) \int_{x-\eta}^{x+\eta} K(x,s)\int_0^a e^{-\lambda_0(a-\delta)} \\&\mathscr{T}(\delta,a)\phi(\delta,s)d\delta dsda+o(e^{(\lambda_0-\varepsilon)t})
  \end{align*}
  where $C_{\lambda_0}=\underset{\lambda\rightarrow \lambda_0}\lim(\lambda-\lambda_0)(\mathbb{I}-\mathscr{B}_{\lambda})^{-1}$ and $\varepsilon$ is any positive number such that $\sigma(\mathbb{A})\cap\{\lambda|\lambda_0-\varepsilon\leq Re \lambda\leq \lambda_0\}=\lambda_0$ holds.
\end{description}
\end{theorem}

\begin{proof}
(1) It will be done in two steps: (i) prove that $\mathbb{A}$ has only one real eigenvalue $\lambda_0$ and $\lambda_0$ is larger than any real part of the other eigenvalues; (ii) prove that $\lambda_0$ is algebraically simple by showing $T(t)$ is compact for $t\ge a_{\dag}$.

(i) Define
$$E=\{\phi \in L^{2}([0,24])|   \int_{x-\eta}^{x+\eta} K(x,s)\phi(s)ds \geq C \phi(x)\},$$
where $C>0$ is a sufficiently small constant.

Recall $\mathscr{F}_{\lambda}$ in \eqref{eq66} and denote the restrictions of
$\mathscr{B}_{\lambda}$,  $\mathscr{F}_{\lambda}$ on $E$ by $\overline{\mathscr{B}_{\lambda}},\overline{\mathscr{F}_{\lambda}}$ respectively.
Then from \eqref{eq:33} and \eqref{eq66},
$$
\overline{\mathscr{B}_{\lambda}}\geq \overline{\mathscr{F}_{\lambda}}.
$$
Given any nonnegative function $\phi(x)$, $\psi(x) \in L^{2}([0,24])$, both not identical to zero, then from
\cite{Amann} and \cite{Nagel}, $\langle e^{\mathbb{B}a}\phi,\psi\rangle>0$ for all $a>0$.
From the expression of $\mathscr{B}_{\lambda}$ and $K(x,s)$, it follows that
\begin{align} \label{snp}
\langle \mathscr{B}_{\lambda}\phi,\psi \rangle >0,\ \text{for all real}\ \lambda>0.
\end{align}
Furthermore, if $\phi(x)\in E$, from assumption $(J1)$, $(J2)$ and the expression of $\mathscr{F}_{\lambda}$, we know that
\begin{align*}
\langle \overline{\mathscr{B}}_{\lambda}\phi,\psi \rangle\geq \langle \overline{\mathscr{F}}_{\lambda} \phi,\psi \rangle >0,\ \text{for all real}\ \lambda>0.
\end{align*}

From Lemma \ref{lemma2.1}, there is a $\widetilde{\lambda}_0$ such that $\gamma(\mathscr{F}_{\widetilde{{\lambda}}_0})=1$ and $1$ is an eigenvalue of
${\mathscr{F}_{\widetilde{{\lambda}}_0}}$ with the eigenfunction $\phi_0(x)$.
Remembering that $\phi_0(x)$ is a positive constant, it is easy to check that $\phi_0(x)\in E$, even if it means reducing $C$. Hence, $\overline{\mathscr{F}}_{\widetilde{{\lambda}}_0} \phi_0=\mathscr{F}_{\widetilde{{\lambda}}_0} \phi_0=\phi_0$, which implies $\gamma(\overline{\mathscr{F}}_{\widetilde{{\lambda}}_0})\geq 1$. Moreover, since $\gamma(\overline{\mathscr{F}}_{\widetilde{\lambda}_0})\leq \gamma(\mathscr{F}_{\widetilde{{\lambda}}_0})=1$, one obtains that $\gamma(\overline{\mathscr{F}}_{\widetilde{{\lambda}}_0})=1$. Therefore we conclude that  $$\gamma(\mathscr{B}_{\widetilde{{\lambda}}_0})\geq \gamma(\overline{\mathscr{B}}_{\widetilde{{\lambda}}_0})\geq \gamma(\overline{\mathscr{F}}_{\widetilde{{\lambda}}_0})=1.$$

On the other hand, $\lim_{\lambda\rightarrow +\infty} \gamma(\mathscr{B}_{\lambda})=0$ and hence by continuity there exists a real $\lambda_0$ such that $\gamma(\mathscr{B}_{\lambda_0})=1$.
Since $\mathscr{B}_{\lambda_0}$ is a compact positive operator, by Krein-Rutman Theorem there exists a nonnegative $\phi_{\lambda_0}(x)\in L^2(0,24)$ such that
\begin{equation}\label {eq:52}
\mathscr{B}_{\lambda_0} \phi_{\lambda_0}(x)=\phi_{\lambda_0}(x),
\end{equation}
i.e., $\sigma(\mathscr{B}_{\lambda_0})\neq \varnothing$. Since \eqref{snp}, the operator $\mathscr{B}_{\lambda}$ is semi-nonsupporting. From Theorem 4.3 of \cite{IVO}, we learn that $\gamma(\mathscr{B}_{\lambda})$ is strictly monotone decreasing with respect to real $\lambda$. This is equivalent to the uniqueness of the real eigenvalue of operator $\mathbb{A}$. That is, $\sigma(\mathbb{A})\neq \varnothing$.

When $\lambda>\lambda_0$ and $\gamma(\mathscr{B}_{\lambda})<\gamma(\mathscr{B}_{\lambda_0})=1$, $(\mathbb{I}-\mathscr{B}_{\lambda})^{-1}$ exists and is positive, and hence $R(\lambda,\mathbb{A})$ is positive from \eqref{eq:5555}. Thus, $\lambda_0$ is larger than any real part of the other eigenvalues.

(ii) Integrating along the characteristic, we obtain
\begin{eqnarray*}
p(a,t,x)=\left\{\begin{aligned}
&\mathscr{T}(a-t,0,t) p_0(a-t,x),\ a\geq t,\\
&\mathscr{T}(0,0,a) \int_0^{a_{\dag}} \beta(a)\int_{x-\eta}^{x+\eta} K(x,s)p(a,t-a,s) d s da,\ a<t.
\end{aligned}
\right.
\end{eqnarray*}
When $t\geq a_{\dag}$,
$$T(t)\phi(a,x)=\mathscr{T}(0,0,a)\int_{0}^{a_{\dag}}\beta(a)\int_{x-\eta}^{x+\eta} K(x,s)[T(t-a)\phi](a,s)ds d a.$$
Let $\phi_n$ weakly converge to $\phi$ in $X$. By the compactness of $\mathscr{T}(0,0,a)$, one has
$$\|\mathscr{T}(0,0,a)\int_{0}^{a_{\dag}}\beta(a)\int_{x-\eta}^{x+\eta} K(x,s)[T(t-a)(\phi_n-\phi)](a,s)dsda\|_{L^2([0,24])}\rightarrow 0.$$
On the other hand,
\begin{align*}
&\|\mathscr{T}(0,0,a)\int_0^{a_{\dag}} \beta(a) \int_{x-\eta}^{x+\eta} K(x,s)[T(t-a)(\phi_n-\phi)](a,s)dsda\|_{L^2([0,24])}\nonumber\\
&\leq \|\mathscr{T}(0,0,a)\| \|\int_0^{a_{\dag}} \beta(a) \int_{x-\eta}^{x+\eta} K(x,s)[T(t-a)(\phi_n-\phi)](a,s)dsda\|_{L^2([0,24])}\nonumber\\
&\leq M\|\phi_n-\phi\|_{L^2([0,24])}
\end{align*}
is bounded. Using the dominant convergence theorem, we get $\displaystyle\lim_{n\rightarrow \infty} \|T(t)(\phi_n-\phi)\|=0$; that is, $T(t)\phi_n$ converge strongly to $T(t)\phi$. Thus, $T(t)$ is compact.

By the results of \cite{Clement}, the semigroup $T(t)$ generated by $\mathbb{A}$, is a positive semigroup and
$$\lambda_0=s(\mathbb{A})=\omega_0(\mathbb{A})$$
where $s(\mathbb{A})$, $\omega_0(\mathbb{A})$ denote the spectral bound of $\mathbb{A}$ and the growth bound of the semigroup $T(t)$ respectively.
Since $T(t)$ is compact, it is known from  \cite{Clement} that $\omega_{ess}(\mathbb{A})=-\infty$.
Furthermore, from Theorem 9.10 in \cite{Clement}, it is easy to get that
$$\lambda_0=\{\lambda | Re \lambda = s(\mathbb{A})\}.$$
It means that $\lambda_{0}$ is a pole of the resolvent of $R(\lambda,\mathbb{A})$. Thus, $\gamma(\mathscr{B}_{\lambda_0})=1$ is a pole of $R(\lambda,\mathscr{B}_{\lambda_0})$. Moreover, by \eqref{snp}, one obtains that $\mathscr{B}_{\lambda_0}$ is a non-semisupporting operator. Since Theorem 1 in \cite{SAW}, one can obtain
that $\gamma(\mathscr{B}_{\lambda_0})=1$ is an algebraically simple eigenvalue of $\mathscr{B}_{\lambda_0}$. This is equivalent of
$\lambda_{0}$ being an algebraically simple eigenvalue of $\mathbb{A}$.

(2) From (1), we have that $\sigma(\mathbb{A})\cap\{\lambda|\lambda_0-\varepsilon\leq Re \lambda\leq \lambda_0\}=\lambda_0$, and $T(t)$ is a compact operator. Then from Theorem 5 of \cite{Yu}, there are constants $C$ and $T_{0}$, such that
\begin{align*}
\|T(t)-T(t)P_{\lambda_{0}}\|\leq Ce^{^{(\lambda_{0}-\epsilon)t}},  t\geq T_{0},
\end{align*}
where $T(t)$ is the semigroup generated by $\mathbb{A}$, $P_{\lambda_{0}}$ is the mapping from $X$ to $B_{\lambda_{0}}$, and
$B_{\lambda_{0}}$ is the eigenvalue space of $\lambda_{0}$ of $\mathbb{A}$. Furthermore,
\begin{align}\label {eq:77}
T(t)\phi= T(t)P_{\lambda_{0}}\phi+o(e^{^{(\lambda_{0}-\epsilon)t}}).
\end{align}
Since $\lambda_{0}$ is an algebraically simple eigenvalue of $\mathbb{A}$, then it is known from \cite{Hille} that
\begin{align}\label {eq:88}
P_{\lambda_{0}}\phi=\underset{\lambda\rightarrow \lambda_0}\lim(\lambda-\lambda_{0})R(\lambda,\mathbb{A})\phi.
\end{align}
Combining \eqref{eq:77} and \eqref{eq:88},
\begin{align*}\label {eq:99}
T(t)\phi= e^{\lambda_{0}t}\underset{\lambda\rightarrow \lambda_0}\lim(\lambda-\lambda_{0})R(\lambda,\mathbb{A})\phi+o(e^{^{(\lambda_{0}-\epsilon)t}}).
\end{align*}
Then, using the expression \eqref{eq:5555} of $R(\lambda,\mathbb{A})\phi$,
\begin{align*}
T(t)\phi(a,x)=&e^{\lambda_0 t}e^{-\lambda_0 a} \mathscr{T}(0,a) C_{\lambda_0} \int_0^{a_{\dag}} \beta(a) \int_{x-\eta}^{x+\eta} K(x,s)\int_0^a e^{-\lambda_0(a-\delta)} \mathscr{T}(\delta,a) \\&\phi(\delta,s)d\delta dsda+o(e^{(\lambda_0-\varepsilon)t}).
\end{align*}
\end{proof}

\begin{remark}
Here, we can see that Theorem \ref{th2} is a direct result of Theorem \ref{66666}, so the proof of Theorem \ref{th2} is complete.
\end{remark}

\section{Existence of steady states}
\noindent
As for the steady states \eqref{eq:04} satisfying \eqref{eq:05}, our main result is Theorem \ref{th3}.
In this section, we prove Theorem \ref{th3} directly according to Theorem \ref{66666}.

\begin{proof} Firstly, let $\lambda_{0}$ be as defined in Theorem \ref{66666}. Then, we look for the steady states \eqref{eq:04} in the following three cases according to the sign of $\lambda_{0}$.

(1)When $\lambda_{0}>0$, we argue this case by a contradiction. Assume that $p_{s}(a,x)$ is a nonnegative solution of \eqref{eq:04} satisfying \eqref{eq:05}.
It is easy to see that $p_{s}(a,x)=p(a,t,x)$ is also a solution of the following system
\begin{equation*}
\left\{ \begin{array}{lll}
Dp(a,t,x)- \delta \Delta p(a,t,x)+\mu (a)p(a,t,x) = 0,&(a,t,x)\in Q_{a_\dagger}, \\
p(a,t,0) =p(a,t,24),&(a,t)\in (0,a_\dagger)\times (0,T), \\
\partial_x p(a,t,0) =\partial_x  p(a,t,24),&(a,t)\in (0,a_\dagger)\times (0,T), \\
\displaystyle p(0,t,x)=\displaystyle \int_0^{a_\dagger} \beta(a)\displaystyle \int_{x-\eta}^{x+\eta}K(x,s)p(a,t,s)ds da,&(t,x)\in (0,T)\times (0,24), \\
\displaystyle p(a,0,x)=p_s(a,x) , &(a,x)\in (0,{a_\dagger} ) \times (0,24).
          \end{array}
  \right. \end{equation*}
Then by a result of Theorem \ref{th2}, one has the asymptotic expression
\begin{align*}
p(a,t,x)
=&e^{\lambda_{0}t}e^{-\lambda_{0}a}\mathscr{T}(0,a)C_{\lambda_{0}}\displaystyle \int_0^{a_\dagger} \beta(a)\displaystyle \int_{x-\eta}^{x+\eta}K(x,s)\\
&\displaystyle \int_0^{a}e^{-\lambda_{0}(a-\sigma)}\mathscr{T}(\sigma,a)p_{0}(\sigma,s)dsdad\sigma+o(e^{(\lambda_{0}-\epsilon)t}).
\end{align*}
Thus,
\begin{align*}
\| p_{s}(a,x) \|_{L^2((0,a_\dagger)\times(0,24))}=\underset{t\rightarrow +\infty}\lim \|p(a,t,x)\|_{L^2((0,a_\dagger)\times(0,24))}=+\infty
\end{align*}
which is a contradiction. Thus, there is no nonnegative solution of \eqref{eq:04} satisfying \eqref{eq:05}.

(2) When $\lambda_{0}=0$, it means that $0\in \sigma(\mathbb{A})$. From the definition of $\mathbb{A}$, every eigenfunction related to $0$ and its multiplications by any constant are solutions of \eqref{eq:04}.

Recalling \eqref{eq:52} from the proof of Theorem \ref{66666}, there is a nonnegative function $\phi_{\lambda_0}(x)\in L^2(0,24)$ such that
$$\mathscr{B}_{\lambda_0}(\phi_{\lambda_0}(x))=  \int_0^{a_\dagger} \beta(a)  \int_{x-\eta}^{x+\eta}K(x,s)e^{-\lambda_{0} a}\mathscr{T}(0,a)\phi_{\lambda_0}(s)ds da=\phi_{\lambda_0}(x).$$
By Lemma \ref{lemma:01}, one knows that $\mathscr{T}(0,a)$ is a bounded operator on $X$. Using Cauchy-Schwarz inequality, for arbitrary $x_{0} \in(0,24)$, one obtains
\begin{align*}
&|\phi_{\lambda_0}(x)-\phi_{\lambda_0}(x_{0})| \\
=&\left| \int_0^{a_\dagger} \beta(a)  \int_{x-\eta}^{x+\eta}K(x,s)e^{-\lambda_{0} a}\mathscr{T}(0,a)\phi_{\lambda_0}(s)ds da \right.\\
&\left. -\int_0^{a_\dagger} \beta(a)  \int_{x_{0}-\eta}^{x_{0}+\eta}K(x_{0},s)e^{-\lambda_{0} a}\mathscr{T}(0,a)\phi_{\lambda_0}(s)ds da \right| \\
\le &\|\beta(a)\|_{L^{\infty}(0,a_{\dagger})} \left|\int_0^{a_\dagger}  \int_{x-\eta}^{x+\eta}K(x,s) \mathscr{T}(0,a)\phi_{\lambda_0}(s)ds da \right.\\
&\left.-\int_0^{a_\dagger}  \int_{x_{0}-\eta}^{x_{0}+\eta}K(x_{0},s)\mathscr{T}(0,a)\phi_{\lambda_0}(s)ds da \right| \\
\leq &\|\beta(a)\|_{L^{\infty}(0,a_{\dagger})} \left| \int_0^{a_\dagger}  \int_{x-\eta}^{x+\eta}(K(x,s)-K(x_{0},s)) \mathscr{T}(0,a)\phi_{\lambda_0}(s)ds da\right|  \\
&+\|\beta(a)\|_{L^{\infty}(0,a_{\dagger})} \left| \int_0^{a_\dagger}  \int_{x_{0}-\eta}^{x+\eta}K(x_{0},s)\mathscr{T}(0,a)\phi_{\lambda_0}(s)ds da\right| \\
&+\|\beta(a)\|_{L^{\infty}(0,a_{\dagger})} \left|\int_0^{a_\dagger} \int_{x_{0}+\eta}^{x+\eta}K(x_{0},s)\mathscr{T}(0,a)\phi_{\lambda_0}(s)ds da \right|\\
\leq & \|\beta(a)\|_{L^{\infty}(0,a_{\dagger})}\|K(x,s)-K(x_{0},s)\|_{L^2(x-\eta,x+\eta)} \|\mathscr{T}(0,a)\phi_{\lambda_0}(s)\|_{L^2((0,a_{\dagger})\times(0,24))}\\
&+\|\beta(a)\|_{L^{\infty}(0,a_{\dagger})} \left(\int_{x_{0}-\eta}^{x+\eta}|K(x_{0},s)|^{2} d s\right)^{\frac{1}{2}}  \|\mathscr{T}(0,a)\phi_{\lambda_0}(s)\|_{L^2((0,a_{\dagger})\times(0,24))}\\
&+\|\beta(a)\|_{L^{\infty}(0,a_{\dagger})} \left(\int_{x_{0}+\eta}^{x+\eta}|K(x_{0},s)|^{2} d s\right)^{\frac{1}{2}}  \|\mathscr{T}(0,a)\phi_{\lambda_0}(s)\|_{L^2((0,a_{\dagger})\times(0,24))}\\
\leq& C\|\beta(a)\|_{L^{\infty}(0,a_{\dagger})}\|K(x,s)-K(x_{0},s)\|_{L^2(x-\eta,x+\eta)} \|\phi_{\lambda_0}(s)\|_{L^2(0,24)}\\
&+C\|\beta(a)\|_{L^{\infty}(0,a_{\dagger})} \left(\int_{x_{0}-\eta}^{x+\eta}|K(x_{0},s)|^{2} d s\right)^{\frac{1}{2}}  \|\phi_{\lambda_0}(s)\|_{L^2(0,24)}\\
&+C\|\beta(a)\|_{L^{\infty}(0,a_{\dagger})} \left(\int_{x_{0}+\eta}^{x+\eta}|K(x_{0},s)|^{2} d s\right)^{\frac{1}{2}}  \|\phi_{\lambda_0}(s)\|_{L^2(0,24)}\\
\rightarrow& 0, \text{ as $x\rightarrow x_0$}.
\end{align*}
Thus, $\phi_{\lambda_0}(x)$ is continuous about $x$.
Then, from the proof of Lemma \ref{lemma3.2},
it is easy to check that
$$\phi(a,x)=\mathscr{T}(0,a)\phi_{\lambda_0}(x)$$
is an eigenfunction of the eigenvalue $\lambda_{0}=0$ of $\mathbb{A}$. Therefore, the steady states are
$$p_s(a,x)=c\mathscr{T}(0,a)\phi_{\lambda_0}(x)\ge 0, \text{ for any constant $c>0$}.$$
By a result of Lemma \ref{lemma:01}, we know that $\mathscr{T}(0,a)$ is strongly continuous with respect to $a$.
Hence, $p_s(a,x)$ is continuous about $a$, $x$ in $(0,a_\dagger)\times (0,24)$.

Consider smooth function $v(a,x)$ such that $v(a,x)=e^{\int_0^{a} \mu(\rho)d\rho} p_s(a,x)\ge 0$ a.e $(a,x)\in(0,a_{\dagger})\times (0,24)$. Then, from \eqref{eq:04}, $v(a,x)$ satisfies
\begin{equation}\label {eq:8888}
\left\{ \begin{array}{lll}
\partial_a v- \delta \Delta v= 0,\hskip 4.29cm (a,x)\in (0,a_{\dagger})\times (0,24), \\
v(a,0)=v(a,24),\hskip 6.29cm a\in (0,a_{\dagger}), \\
\partial_x v(a,0) =\partial_x  v(a,24),\hskip 5.49cm a\in (0,a_{\dagger}), \\
v(0,x)= \int_0^{a_\dagger} \beta(a) \int_{x-\eta}^{x+\eta}K(x,s) e^{-\int_0^{a} \mu(\rho)d\rho}v(a,s)ds da, \ x\in (0,24). \\
\end{array}
\right. \end{equation}
From the strong maximum principle, $v(a,x)>0$ for $(0,a_{\dagger})\times (0,24)$. Then, $v(0,x)=\int_0^{a_\dagger} \beta(a) \int_{x-\eta}^{x+\eta}K(x,s) e^{-\int_0^{a} \mu(\rho)d\rho}v(a,s)ds da>0$ for $x\in (0,24)$. Assume by contradiction that $v$ attains its minimum $0$ at $(a_0,0)$, that is, $v(a_0,0)=0$ for some $a_0\in(0,a_{\dagger})$. Then, $\partial_a v(a_0,0)=0$ and $\partial_x v(a_0,0)\ge 0$. Since $v(a,o)=v(a,24)$ for $a\in (0,a_{\dagger})$, one has that $v(a_0,24)=0$ and $\partial_a v(a_0,24)=0$, $\partial_x v(a_0,24)\le 0$. Since $\partial_x v(a,0)=\partial_x v(a,24)$ for $a\in (0,a_{\dagger})$, we obtain that $\partial_x v(a_0,0)=\partial_x v(a_0,24)=0$. Then, $\Delta v(a_0,0)=\partial_{xx} v(a_0,0)>0$ since $v(a,x)>0$ for $(0,a_{\dagger})\times (0,24)$. Thus, $(\partial_a v-\delta \Delta v)(a_0,0)<0$ which is a contradiction of the first equation of \eqref{eq:8888}. So that, $v(a,0)$, $v(a,24)>0$ for $a\in(0,a_{\dagger})$. By $v(0,x)=\int_0^{a_\dagger} \beta(a) \int_{x-\eta}^{x+\eta}K(x,s) e^{-\int_0^{a} \mu(\rho)d\rho}v(a,s)ds da$, one also has that $v(0,0)$, $v(0,24)>0$. Therefore, we can conclude that for any $a_1<a_{\dagger}$,
$$p_s(a,x)=e^{-\int_0^{a}\mu(\rho)d\rho} v(a,x)>0, \text{ a.e. in $[0,a_1]\times [0,24]$}$$
since $\int_0^a \mu(\rho)d\rho<\infty$ for $a<a_{\dagger}$. Finally, there exists $\rho_0>0$ such that
$$p_s(a,x)\ge \rho_0>0,\ \text{a.e. } (a,x)\in (0,a_1)\times (0,24).$$

(3) When $\lambda_{0}<0$, it follows from the arguments of (1) that
\begin{align*}
\| p_{s}(a,x) \|_{L^2((0,a_\dagger)\times(0,24))}=\underset{t\rightarrow +\infty}\lim \|p(a,t,x)\|_{L^2((0,a_\dagger)\times(0,24))}=0.
\end{align*}
Thus,
\begin{equation*}
p_{s}(a,x)=0 \ \ a.e. \  (a,x)\in (0,{a_{\dagger}})\times (0,24).
\end{equation*}
\end{proof}

\noindent{\bf Acknowledgements:} C.P.F. thanks to 14/07615-3 S\~ao Paulo Research Fundation (FAPESP).


\begin{thebibliography}{99}
\bibitem{Ains1} B.E. Ainseba, M. Langlais. Sur un probl\`eme de contr\^ole d'une population
structur\'ee en \^age et en espace, C. R. Acad. Sci. Paris S\'er. I. 323 (1996) 269--274.
\bibitem{Ains2} B.E. Ainseba, M. Langlais. On a population dynamics control problem with age dependence
and spatial structure, J. Math. Anal. Appl. 248 (2000) 455--474.
\bibitem{Ains3} B.E. Ainseba, S. Ani\c{t}a.  Local exact controllability of the age-dependent population
dynamics with diffusion, Abstr. Appl. Anal. 6 (2001) 357--368.
\bibitem{Amann} H. Amann. Dual semigroups and second-order linear elliptic boundary value problems, Israel. J. Math. 45 (1983) 225--254.
\bibitem{As}Ani\c{t}a S. Analysis and control of age-dependent population dynamics. Springer Science \& Business Media, 2000.
\bibitem{Chan} W.L. Chan, B.Z. Guo. On the semigroups for age-size dependent population dynamics with spatial diffusion,
Manuscripia. Math. 66 (1990) 161--181.
\bibitem{Clement} P. Cl\'{e}ment, H. Heijmans, S. Angenent. One parameter semigroups, CWI Monographs 5 (1987) 1--312.
\bibitem{Cara} E. Fern\'{a}ndez-Cara.  Null controllability of the semilinear heat equation,
ESAIM:COCV 2 (1997) 87--103.
\bibitem{Guo} B.Z. Guo, W.L. Chan. On the semigroup for age dependent population Dynamics with spatial diffusion,
J. Math. Anal. Appl. 184 (1994) 190--199.
\bibitem{Garr}  M.G. Garroni, M. Langlais.  Age dependent population diffusion with external
  constraints, J. Math. Biol. 14 (1982) 77--94.
  \bibitem{Gatton2013} M.L. Gatton, N. Chitnis, T. Churcher, M.J. Donnelly, A.C. Ghani, H.C.J Godfray, F. Gould, I. Hastings, J. Marshall, H. Ranson, M. Rowland, J. Shaman, S.W. Lindsat. The importance of mosquito behavioral adaptations to malaria control in Africa, Evolution 64 (2013) 1218--1230.
\bibitem{Gurt}  M.E. Gurtin.  A system of equations for age dependent population diffusion, J.
Theor. Biol. 40 (1972) 389--392.
\bibitem{Hille} E. Hille, R.S. Phillips. Function analysis and semigroups,  American Mathematical Soc. 31, 1996.
\bibitem{Ladyzenskaja} O.A. Lady\u{z}enskaja, V.A Solonnikov. Linear and quasilinear equations of parabolic type, American Mathematical Soc. 23, 1988.
\bibitem{Langlais} M. Langlais. Large time behavior in a nonlinear age-dependent population dynamics problem with spatial diffusion, J. Math. Biol. 26 (1988) 319--346.
\bibitem{IVO} I. Marek. Frobenius theory of positive operators: Comparison theorems and applications, SIAM J. Appl. Math. 19 (1970) 607--628.
\bibitem{Nagel} R. Nagel, et al. One-parameter semigroups of positive operators, Springer Verla, 1986.
\bibitem{Pazy} A. Pazy. Semigroups of linear operators and applications to partial differential equations,
Springer Science \& Business Media, 2012.
\bibitem{SAW} I. Sawashima. On spectral properties of some positive operators, Nat. Sci. Report Ochanomicu Univ. 15 (1964) 53--64.
\bibitem{Song} J. Song et al. Spectral properties of population operator and asymptotic behaviour of population semigroup, Acta Mathematica Scientia, 2 (1982) 139--148.
\bibitem{Webb} G.F. Webb. Theory of nonlinear age-dependent population dynamics, CRC Press, 1985.
\bibitem{Ye} Q. Ye, Z. Li. Introduction to reaction-diffusion equations, Science Press, 1994.
\bibitem{Yu} J.Y. Yu, B.Z. Guo, G.T. Zhu. Asymptotic expression in $L[0,r_m]$ for population evolution and controllability of population system, J. System Sci. Math. Scis. 7 (1987) 97--104.
\end{thebibliography}
\end{document}